\documentclass[12pt]{amsart}
\usepackage{graphicx}
\usepackage[fontsize=12pt]{scrextend}
\usepackage[]{url}
\usepackage[]{enumitem}
\usepackage{amssymb}
\makeatletter
\renewcommand\@biblabel[1]{}
\makeatother
\newtheorem{theorem}{Theorem}
\newtheorem{lemma}[theorem]{Lemma}

\theoremstyle{definition}
\newtheorem{definition}[theorem]{Definitions}
\newtheorem{notation}[theorem]{Notation}

\newtheorem{remark}[theorem]{Remark}
\numberwithin{theorem}{section}

\newcommand{\abs}[1]{\lvert#1\rvert}

\title[Balls Isoperimetric in $\mathbb{R}^n$ with Volume and Perimeter Densities $r^m$ and $r^k$]{Balls Isoperimetric in $\mathbb{R}^n$ with Volume and Perimeter Densities $r^m$ and $r^k$}
\author[\tiny{L. Di Giosia,  J. Habib,  L. Kenigsberg,  D. Pittman, W. Zhu}]{Leonardo Di Giosia,  Jahangir Habib,  Lea Kenigsberg,  Dylanger Pittman, Weitao Zhu\vspace{-5ex}}

\newlength\tindent
\begin{document}
\begingroup
\def\uppercasenonmath#1{}% this disables uppercasing title
\let\MakeUppercase\relax % this disables uppercasing authors
\maketitle
\endgroup

\begin{abstract}

We have discovered a "little" gap in our proof of the sharp conjecture that in $\mathbb{R}^n$ with volume and perimeter densities $r^m$ and $r^k$, balls about the origin are uniquely isoperimetric if $0 < m \leq k - k/(n+k-1)$, that is, if they are stable (and $m > 0$). The implicit unjustified assumption is that the generating curve is convex.

\end{abstract}

\section{\textbf {Introduction}}

The generalized Log-Convex Density Conjecture says that in $\mathbb{R}^n$ with smooth radially symmetric perimeter and volume densities, balls about the origin are isoperimetric if they are stable [Mo2]. In a recent preprint, Alvino et al. [ABCMP, Conj. 5.1] after Diaz et al. [DHHT, 2.3, 4.21, 4.22] and Carroll et al. [CJQW] feature the case of densities $r^m$ and $r^k$ (we'll assume $m$, $k > 0$), for which the stability condition is $m \leq k - k/(n+k-1)$ (see [Mo5, (1a)], which replaces $n$ by $n+1$). Sean Howe [Ho, Ex. 3.5(4)] had proved the result under the stronger hypothesis $m\leq k - 1$. Our Theorem 3.1 proves the conjecture for this choice of densities, assuming that the generating curve is convex (which is known when $m \leq k-1$). If that assumption could be removed, Theorem $3.1$, together with the $2D$ results of Alvino et al., would complete the proof of a conjecture by Diaz et al. [Di, Conj. 4.22(1)] and hence a following corollary [Di, Cor.4.24]. Alvino et al. [Al, Conj. 8.1 and following remarks] also explain how a related conjecture of Caldiroli and Musina [CM, p. 423] would follow.

\bigskip

\subsection{Acknowledgments}
This paper is the work of the “SMALL” 2016 Geometry Group, advised by Frank Morgan. We would like to thank the NSF, Williams College (including the Finnerty Fund), the MAA, and Stony Brook University for supporting the “SMALL” REU and our travel to MathFest.

\section{\textbf{Existence and Regularity }}

Theorems \ref{existence} and \ref{bound} guarantee the existence and boundedness of isoperimetric regions of all volumes.
\begin{figure}[h!]\label {figurr}

  \includegraphics[width=.9
\textwidth]{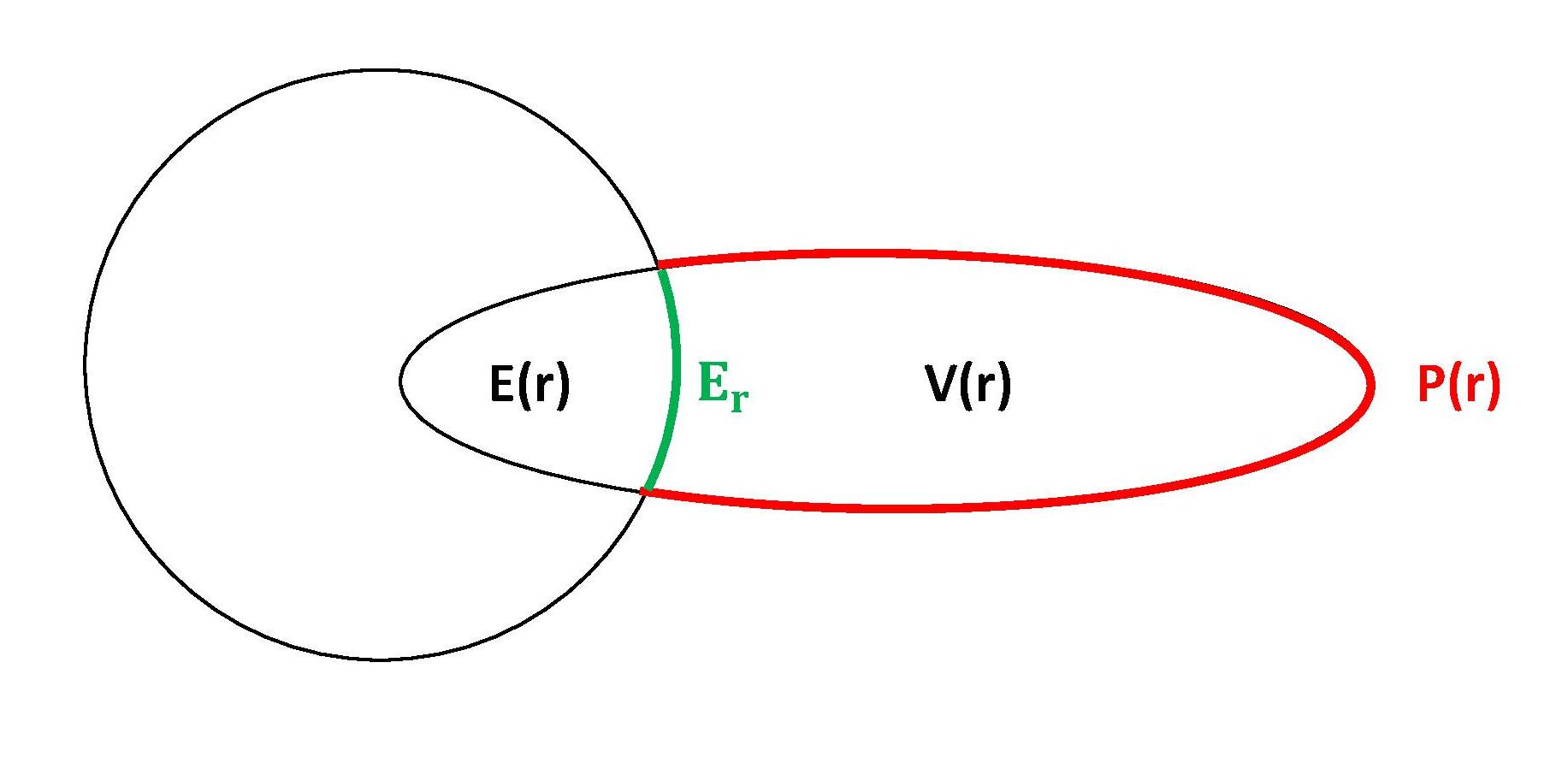}
\caption{By relating the slice $E_r$ of $E$ to the rate of growth of the volume $|E(r)|$ inside the unit sphere, one can obtain inequalities useful in proving existence and boundedness.}
\end{figure}
The following definitions will be used in Theorems \ref{existence} and \ref{bound} and Lemma \ref{projection}. 
\begin{notation}
Denote the sphere and ball of radius $r$ by $S(r)$ and $B(r)$. Let $|S|$ denote the unweighted measure of a surface or region $S$. For a region $E$, let $E_r$ be the slice of $E$ by the sphere $S(r)$ (see [Mo, 4.11]). Let $B_r$ be the restriction of the boundary of $E$ to the exterior of the ball $B(r)$ and define $P(r)$ as $|B_r|$. Let $E(r)$ denote the restriction of $E$ to the ball $B(r)$.
Furthermore, let $V(r)$ be the volume of the restriction of $E$ to the exterior of the ball $B(r)$. Let  $p(r)$ denote the perimeter of $E_r$. In the presence of any volume density $f$, apply $f$ as a subscript to indicate $f$ weighted volume. Similarly for any perimeter density $g$, apply $g$ as a subscript to indicate $g$ weighted perimeter. In particular, $P_g(r)$ denotes the weighted measure of $B_r$ and $P_g(0)$ denotes the weighted perimeter of E.

\end{notation}

\begin{lemma}\label{projection}
Let $E$ be a region in  $\mathbb{R}^n$ with continuous, nondecreasing, and radial volume density $f$. If $E$ contains finite weighted volume then for any $r$, $$P(r)\geq |E_r|.$$

\end{lemma}

\begin{proof}
Since the radial projection $\pi_r$ of the exterior onto $S_r$ is area nonincreasing, it suffices to show that $\pi_r(B_r)$ covers $E_r$, up to a set of measure zero. If not, then a set of positive measure in $E_r$ is not in the projection. The  product of this set with  $(r, \infty)$ is contained in $E$. But this set must have infinite weighted volume, violating our assumption of finite weighted volume.
\end{proof}

\begin{theorem}\label{existence}

In $\mathbb{R}^n$ with nondecreasing, divergent, radial, continuous volume and perimeter densities $f(r)$ and $g(r)$, if $f(r)\leq cg(r)$ outside a bounded set, then for every given volume there exists a perimeter-minimizing region.

\end{theorem}

\begin{proof}
Since the case $n=1$ is trivial, we may assume that $n \geq 2$.
Consider a sequence of regions of the prescribed volume with perimeter tending to the infimum. By compactness [Mo1, 9.1] we may assume convergence to a perimeter-minimizing region. The difficulty is that the enclosed volume may be strictly less than the prescribed volume, that some volume disappears to infinity. In that case we may assume that for some $\epsilon > 0$, for all $0<R$, for a tail of the sequence, the volume outside the ball of radius $R$ about the origin is at least $\epsilon$. Fix $R$ arbitrarily large and such a region $E$. Recall that $|E_r|$ is the unweighted area of the slice of E by an \textit{n}-sphere of radius $r$. By Morgan [Mo, \textsection4.11], for almost all $r$,

$$ |E_r|f(r) = |E_r|_f = -\frac{\partial V_f}{\partial r}.$$ Therefore

$$ \int_R^{\infty}cg(r)|E_r|\,dr  \geq \int_R^{\infty}f(r)|E_r|\,dr = V_f(R) \geq \epsilon.$$ 
Let $$ M = \sup \{|E_r|:r\geq R\}.$$
By Lemma \ref{projection} $$P(r) \geq |E_r|. $$ Hence
\begin{equation} \label{sock}
P_g(0) \geq P_g(r) \geq P(r)g(R) \geq Mg(R).
\end{equation}
Since $g(R)$ may be arbitrarily large, we may assume that $M$ and hence $|E_r|$ is small. Hence, by the standard isoperimetric inequality on $S^{n-1}$, for all $r>R$ $$p(r) \geq c_n |E_r| ^{\frac{n-2}{n-1}}$$ for a suitable dimensional constant $c_n$. Note that $P_g(r)$ is nonincreasing and hence differentiable almost everywhere. By Morgan [Mo, \textsection4.11], for almost all $r$,

$$-\frac{\partial P_g(r)}{\partial r} \geq g(r)p(r).$$ Hence,

$$P_g(0) \geq P_g(R) \geq \int_R^\infty -\frac{\partial P_g(r)}{\partial r}dr \geq \int_R^\infty p(r)g(r)dr \geq \frac{1}{c}\int_R^\infty  p(r)f(r)dr$$ $$ \geq \frac{c_n}{c} \int_R^\infty |E_r| ^{\frac{n-2}{n-1}}f(r)dr \geq \frac{c}{M^{\frac{1}{n-1}}}\int_R^\infty|E_r| f(r)dr\geq \frac{c}{M^{\frac{1}{n-1}}}\epsilon, $$ where $c$ may change from line to line. Note that the second inequality follows from the fact that $P_g(R)$ is a nonincreasing function so that any jump discontinuities are decreases. Combining the above with (\ref{sock}) we obtain $$P_g(0)^{\frac{n}{n-1}}\geq c g(R)^{\frac{1}{n-1}}\epsilon.$$ Since $P_g(0)^{\frac{n}{n-1}}$ is an element of a bounded sequence, it is bounded above; but $g(R)^{\frac{1}{n-1}}$ can be made arbitrarily large, a contradiction. Therefore there is no loss of volume to infinity and the limit provides the desired perimeter-minimizing region.

\end{proof}

\begin{theorem} \label{bound}
Consider continuous  radial nondecreasing perimeter and volume densities $g$ and $f$ in $\mathbb{R}^n$, with $g \slash f$ nondecreasing. Then isoperimetric sets are bounded. 
\end{theorem}

\begin{proof}
Let $E$ be an isoperimetric set and assume that it is not bounded. 
By Lemma \ref{projection} 
\begin{equation} \label{ineq1}
P_g(r) \geq |E_r|_g.
\end{equation}
The standard isoperimetric inequality on the sphere tells us that for any $r$ for which $|E_r|$ is at most half of the area of $S(r)$, one has
\begin{equation}\label{ineq2}
p(r)\geq c_n|E_r|^{\frac{n-2}{n-1}},
\end{equation} where $c_n$ is some dimensional constant. In turn, if E has bounded perimeter by Lemma \ref{sock}, 
\begin{equation}\label{ineq3}
|E_r| \leq \frac{1}{2}|S_r|
\end{equation} for all r big enough. Recalling that in $E_r$ the density has the constant perimeter value $g(r)$, inequality (\ref{ineq2}) is equivalent to

\begin{equation}\label{ineq4}
p_g(r)\geq c_n|E_r|_g^{\frac{n-2}{n-1}}g(r)^{\frac{1}{n-1}} .
\end{equation}
Noting that $f$ is bounded below and rewriting $g=gf\slash f$, we combine the lower bound of $f(r)^{\frac{1}{n-1}}$ with the constant $c_n$. Together with (\ref{ineq1}) we obtain
\begin{equation}\label{ineq5}
\begin{split}
p_g(r)\geq c_n\left(\frac{g}{f}(r)\right)^{\frac{1}{n-1}}|E_r|_g^{\frac{n-2}{n-1}} &= c_n\left(\frac{g}{f}(r)\right)^{\frac{1}{n-1}}|E_r|_g^{\frac{-1}{n-1}}|E_r|_g \\ &\geq c_n(\frac{g}{f}(r))^{\frac{1}{n-1}}|E_r|_g P_g(r)^{\frac{-1}{n-1}},
\end{split}
\end{equation} where $c_n$ may change from line to line. Note that we can take a negative power of $P_g(r)$ since we assumed it is never zero. Moreover, note that we took negative powers of $|E_r|_g$, which may be zero if $E$ is disconnected. However, for such values of $r$, every quantity in the inequality is zero and so the statements hold trivially.
Because $V(r)$, $P(r)$, and $g$ are monotonic functions, $V_f(r)$ and $P_g(r)$ are differentiable almost everywhere. By Morgan [Mo1, \textsection 4.11] 
\begin{equation}
-\frac{\partial P_g(r)}{\partial r} = \bigg|\frac{\partial P_g(r)}{\partial r}\bigg| \geq p_g(r)
\end{equation}
and  
\begin{equation}
-\frac{g}{f}(r)\frac{\partial V_f(r)}{\partial r} = \frac{g}{f}(r)|E_r|_f = |E_r|_g = -\frac{\partial V_g(r)}{\partial r}
\end{equation}
almost everywhere.
Hence we can rewrite (\ref{ineq5}) as

\begin{equation}
\begin{split}
-\frac{\partial P_g(r)}{\partial r} &\geq c_n\left(\frac{g}{f}(r)\right)^{\frac{1}{n-1}}|E_r|_g P_g(r)^{\frac{-1}{n-1}} =c_n\left(\frac{g}{f}(r)\right)^{\frac{n}{n-1}}|E_r|_f P_g(r)^{\frac{-1}{n-1}} \\ &= -c_n\frac{\partial V_f(r)}{\partial r}\left(\frac{g}{f}(r)\right)^{\frac{n}{n-1}} {P_{g}(r)}^{\frac{-1}{n-1}}. 
\end{split}
\end{equation}
Therefore

\begin{equation}
-\frac{\partial P_g(r)^{\frac{n}{n-1}}}{\partial r} \geq  -c_n\frac{\partial V_f(r)}{\partial r}\left(\frac{g}{f}\right)^{\frac{n}{n-1}}.
\end{equation}
Note that $P_g(r)$ may be discontinuous on a set of measure zero, but it is a monotonic function which converges to zero at infinity. Hence, 
$$ P_g(r)^{\frac{n}{n-1}} \geq \int_r^\infty -\frac{\partial}{\partial r}P_g(r)^{\frac{n}{n-1}}dr.$$ 
Since $V_f(r)$ is a continuous function that converges to zero at infinity, $V_f(r) = \int_r^\infty-\partial V_f(r)\slash \partial r.$
Since $g\slash f$ is a nondecreasing function, integration over $r$
yields
\begin{equation}
P_g(r)^{\frac{n}{n-1}} \geq \int_r^\infty -\frac{\partial}{\partial r}P_g(r)^{\frac{n}{n-1}}dr \geq c_n \left(\frac{g}{f}(r)\right)^{\frac{n}{n-1}}V_f(r).
\end{equation}
Pick a large constant $r_0 \in \mathbb{R}$ such that in particular the restriction $E(r_0) \neq \emptyset$. Let $H$ denote the unnormalized generalized mean curvature, equal to $ dP_g \slash dV_f$.  As such, it is possible to choose a set $E_{\epsilon}$, for $0< \epsilon < \tilde \epsilon $, so that $E_{\epsilon}$  agrees with $E$ outside a ball of radius $r_0$ and whose $f$-weighted volume exceeds that of $E$ by $\epsilon$. Then
$$P_g(E_{\epsilon}) \leq P_g(E) + \epsilon (H(E)+1).$$
Since $V_f(r)$ is continuous and converges to zero at infinity, we can choose $r>r_0$, so that $\epsilon = V_f(r) < \tilde \epsilon $. Let $\bar{E}$ be the restriction of $E_{\epsilon}$ to $B(r)$, then of course the $f$-weighted volume of $\bar{E}$ agrees with that of $E$, and moreover

\begin{equation}
\begin{split}
P_g(\bar E)&= P_g(E_{\epsilon})-P_g(r)+|E_r|_g \\ &\leq  P_g(E) + \epsilon (\frac{1}{n} H(E)+1)-c_n\frac{g}{f}(r)V_f(r)^{\frac{n-1}{n}} + |E_r|_g. \\&=  P_g(E) + \epsilon (\frac{1}{n} H(E)+1)-c_n\frac{g}{f}(r)\epsilon^{\frac{n-1}{n}} + \frac{g}{f}(r)|E_r|_f.
\end{split}
\end{equation}
Since $E$ is isoperimetric, $P_g(\bar E) - P_g( E) \geq 0$, hence
\begin{equation}
0 \leq \epsilon (\frac{1}{n} H(E)+1)-c_n\frac{g}{f}(r)\epsilon^{\frac{n-1}{n}} + \frac{g}{f}(r)|E_r|_f.
\end{equation}
Since we can choose $\epsilon$ arbitrarily small, and $\epsilon$ goes to $0$ much faster than $\epsilon ^{\frac{n-1}{n}}$  we deduce
\begin{equation}
 |E_r|_f \geq c_n\epsilon ^{\frac{n-1}{n}}. 
\end{equation}
That is,
\begin{equation}
 -\frac{\partial V_f(r)}{\partial r} \geq c_nV_f(r)^{\frac{n-1}{n}}
\end{equation}
or equivalently
\begin{equation}
-\frac{\partial V_f(r)^{\frac{1}{n}}}{\partial r} \geq c_n,
\end{equation}
which is a contradiction since we assumed $V(r) > 0$ for all $r$.

\end{proof}

\begin{remark}\label{SER}
Isoperimetric boundaries in a smooth \textit{n}D Riemannian manifold with smooth positive densities are smooth except for a singular set of dimension at most $n-8$ ([Mo4,  Cor. 3.8, Rmk. 3.10], which does not really require volume and perimeter densities $f$ and $g$ equal). Equilibrium implies constant generalized mean curvature

\begin{equation}\label{eq:gencur}
 H_{f,g} = \frac{g}{f} H + \frac{1}{f}\frac{\partial g}{\partial n}
\end{equation}[Mo5]. Here H is the standard inward unnormalized mean curvature, and n is the outward normal. Balls about the origin are stable (have nonnegative second variation) if and only if

\begin{equation}\label{stability}
nr(f/g)(g/f)' + r^2 g''/g - r^2 (1/fg)(f'g') \ge 0.
\end{equation}  
[Mo5]. In particular, for volume and perimeter densities $f = r^m$ and $g = r^k$, balls about the origin are stable if and only if
\begin{equation}\label{stabilitymk}
m \leq k - \frac{k}{k+n-1}
\end{equation}
[Mo5, (1a)].

\end{remark}

\section{\textbf{Isoperimetric Regions}}

Our main Theorem \ref{extension} proves the conjecture of Alvino et al. [Al, Conj. 5.1] that balls are isoperimetric if stable (Rmk. \ref{SER}), under the assumption that a generating curve is convex (Chambers [Ch, Prop. 4.1(1)], which follows from an argument of Morgan and Pratelli that works only if $m \leq k-1$). The proof consists of requisite extensions of the lemmas of Chambers [Ch].

\begin{theorem}\label{extension}
In $\mathbb{R}^n$ with volume density $r^m$ and perimeter density $r^k$ $(m, k > 0)$, balls about the origin are isoperimetric if they have nonnegative second variation, i.e., if $m \leq k - k/(n+k-1)$, assuming that a generating curve is convex.
\end{theorem}

\begin{proof}
Given existence Theorem \ref{existence}, boundedness Theorem \ref{bound}, and regularity Remark \ref{SER}  most of the proof follows Chambers [Ch] word for word. He has just six lemmas which use his assumption that the volume and perimeter densities are equal and log convex: his Lemmas 3.4, 3.5, 3.7. 3.14. 5.1, and 5.2. Our Lemmas 3.5-3.10 will complement and provide the requisite versions of his lemmas.
\end{proof}
\begin{definition}
Let $n$ denote the outward normal. Let $H$ denote the mean curvature, by Chambers' convention the standard inward unnormalized mean curvature. Then the generalized mean curvature (Rmk. \ref{SER}) is given by

	$$H_f,g = H_0 + H_1 = \frac{g}{f} H + \frac{1}{f}\frac{\partial g}{\partial n}.$$
Given a circle of radius $r>0$, centered at the point $(a,0)$, and parameterized by arclength s, let

	$$R(s) = \abs{(a + r \cos \frac{s}{r}, r \sin\frac{s}{r})}.$$
Chambers defines $B_R(f)$ as the centered ball on which the density equals its value at the origin. For our increasing densities, $B_R(f)$ is empty.

\bigskip 
 
An isoperimetric surface of revolution is generated by a planar curve $\gamma$ parametrized by arclength s, with curvature $\kappa(s)$. Following Chambers [Ch, Defn. 3.1], let $C_s$ be the oriented circle centered on the $e_1$ axis that agrees with $\gamma$ to first order at $\gamma(s)$. Let $A_s$ be the oriented circle tangent to  $\gamma$ at $s$, whose signed curvature is equal to $\kappa(s)$. \\

\end{definition}
%Lemma \ref{range} and Lemma \ref{H1''} give the upper bound on $k-m$ for which the 
Our Lemmas \ref{poscurv}, \ref{H''0} and \ref{(-1,0)}, generalizing Chambers [Ch, Lemmas 3.4, 3.5, 3.7] require the following Lemmas \ref{range}  - \ref{H1''}.

%show that this lower bound agrees with the stability condition, and that it is both necessary and sufficient. 
\begin{lemma} \label{range}
Fix an integer $n>1$. For $m,k >0$, 

$$(k-m)(n-1)\frac{1}{r} + \frac{k(k-m-2)(r+a\cos\frac{s}{r})}{R(s)^2} + \frac{k}{r} \geq 0 $$   $$\textit{for all }  r>0, a \geq 0, \textit{and } s \in [0,\pi r \slash 2)$$ if and only if $$0<m\leq k-k\slash (k+n-1).$$
\end{lemma}

\begin{proof}
The inequality holds if and only if

\begin{align*}
k-m &\geq \frac{2\frac{k(r+acos\frac{s}{r})}{R(s)^2} -  \frac{k}{r}}{\frac{n-1}{r} +\frac{k(r+acos\frac{s}{r})}{R(s)^2}}\\
&= 2 - \frac{(2n-2+k)R(s)^2}{R(s)^2 (n-1)+k(r^2+\arccos\frac{s}{r})}.
\end{align*}
Since the expression on the right is maximal when $a=0$, we get the equivalent inequality: $$k-m \geq \frac{k}{n-1+k}.$$ 

\end{proof}

\begin{lemma}\label{H1''}
Fix an integer $n>1$. For $m,k >0$,

$$\frac{a}{r(a+r)^{2}}(n-1)(k-m) +\frac{ak}{(a+r)^{2}r^2}+\frac{ak}{(a+r)^{3}r}(k-m-2) \geq 0 $$
$$\textit{for all }  r>0, a \geq 0, \epsilon \geq 0, \textit{and } s=0$$ if and only if $$0<m\leq k-k\slash (k+n-1).$$
\end{lemma}

\begin{proof}

The inequality holds if and only if
$$k-m \geq \frac{-\frac{k}{r} + \frac{2k}{a+r}}{\frac{n-1}{r}+\frac{k}{a+r}}= 2 - \frac{k+2n -2}{k(\frac{r}{a+r})+n-1}.$$ 
Furthermore, since the expression on the right is maximal when $a=0$, we get the equivalent inequality: $$k-m \geq \frac{k}{n-1+k}.$$ 
\end{proof}

\begin{lemma}[{cf. [Ch, Lemmas 5.1 and 5.2 ]}] \label{H1}
Consider a circle of radius $r>0$ centered at the point $(a,0), a \geq 0$. Consider a counterclockwise parametrization by arclength

$$\alpha(s)=(a+r\cos{\frac{s}{r}}, r\sin{\frac{s}{r}}),$$ where $s \in [0,2\pi r)$. Then for any $s \in [-\pi r\slash 2, \pi r\slash 2]$ $$ H_1'(s)= -\frac{ak}{r}\sin\frac{s}{r}R(s)^{k-m-2}, $$ $$- a k(k-m-2)R(s)^{k-m-4}(r + a\cos\frac{s}{r})\sin\frac{s}{r}$$  
and $$\frac{d}{ds} R(s)^{k-m} = -R(s)^{k-m-2}\sin \frac{s}{r}. $$ Moreover, 
$$H_1''(0)= -\frac{ak}{r^2}(a+r)^{k-m-2}-\frac{ak}{r}(k-m-2)(r+a)^3$$ and  

$$ \frac{d^2}{ds^2} R(s)^{k-m}(0) = -\frac{1}{r}(a+r)^{k-m-2}. $$
\end{lemma}
\begin{proof}
For any $s \in [0,\pi r\slash 2)$ $\gamma(s) \neq (0,0)$, hence $$H_1=\frac{g'(R(s))(N(s)\cdot n(s))}{ f(R(s))}.$$
By direct calculation

$$N=\frac{(a+r\cos\frac{s}{r}, r\sin\frac{s}{r})}{R(s)}, $$ $$ n(s)=(cos\frac{s}{r}, sin\frac{s}{r}),$$
$$N(s)\cdot n(s) = \frac{r + a\cos(\frac{s}{r})}{R(s)},$$

$$H_1= \frac{g'(R(s))}{R(s)f(R(s))}(r + a\cos\frac{s}{r}),$$
and
$$ H_1'(s)= -\frac{ak}{r}\sin\frac{s}{r}R(s)^{k-m-2} $$ $$- a k(k-m-2)R(s)^{k-m-4}(r + a\cos\frac{s}{r})\sin\frac{s}{r}.$$ 
The facts that $$\frac{d}{ds} R(s)^{k-m} = -\frac{a}{r}R(s)^{k-m-2}\sin \frac{s}{r},$$ 
$$H_1''(0)= -\frac{ak}{r^2}(a+r)^{k-m-2}-\frac{ak}{r}(k-m-2)(r+a)^3$$ and  

$$ \frac{d^2}{ds^2} R(s)^{k-m}(0) = -\frac{a}{r}(a+r)^{k-m-2},$$ follow from  straightforward computations.\\

\end{proof}
Lemmas \ref{poscurv}, \ref{H''0} and \ref{(-1,0)}, are the analogs of Chambers [Ch, Lemmas $3.4$, $3.5$, $3.7$]
\begin{lemma}[{cf. [Ch, Lemma 3.4]}] \label{poscurv}
Given a point $s \in [0,\beta)$, $\kappa'(s)\geq 0$ if the following properties hold:
\begin{enumerate}
\item $\gamma'(s)$ is in the second quadrant;
\item $\kappa(s)=\kappa(C_s) > 0$.
\end{enumerate}
If in addition $\gamma'(s) \neq (0,1)$ and $C_s$ is not centered at the origin, then $\kappa'(s) >0$.

\end{lemma}

\begin{proof}
By $(2)$, the circle $A_s$ is simply $C_s$. Since $\kappa(s)>0$ the circle has finite positive radius. As such, $A_s$ approximates $\gamma$ up to  third order at $s$,  hence 

$$\kappa'(C_s)= \kappa'(C_{\tilde s}) = 0,$$

$$H_1'(s)=H_1(\tilde s),$$
and $$ \frac{d}{ds}\frac{g(s)}{f( s)} =\frac{d}{d\tilde s} \frac{g(\tilde s)}{f(\tilde s)}.$$
Since $\kappa(s)=\kappa(C_s)$, by the above equations, $$\frac{d}{ds}H_0(s)=\frac{d}{d\tilde s}\frac{g(\tilde s)}{f( \tilde s)}(n-1)\kappa(s)+ \frac{g(s)}{f( s)}\kappa'(s).$$ 
By Lemma \ref{H1} 
$$H_{f,g}' = -a(k-m)(n-1)\kappa(s)R(s)^{k-m-2}\sin\frac{s}{r} $$ $$+ \kappa'(s)R(s)^{k-m} -\frac{a}{r}\sin\frac{s}{r}R(s)^{k-m-2} $$ $$- a(k-m-2)R(s)^{k-m-4}(r + a\cos\frac{s}{r})\sin\frac{s}{r}.$$ 
Note that $H_{f,g}' =0$, and $\kappa(s)=\kappa(C_s) =\kappa(C_{\tilde s})=1/r.$ Hence,  $$\kappa'(s)= a\sin\frac{s}{r}R(s)^{-2}[(k-m)(n-1)\frac{1}{r} $$ $$+ \frac{k(k-m-2)(r+a\cos\frac{s}{r})}{R(s)^2}+ \frac{k}{r}].$$ 
Since $$a\sin(s\slash r)R(s)^{-3} \geq 0$$ for $s \in [0 \pi r \slash 2)$, and $$m \leq k-k\slash (k+n-1)$$ by Lemma \ref{range}, $\kappa'(s) \geq 0$. 

If $\gamma'(s) \neq (0,1)$ and $C_s$ is not centered at the origin, then $a>0$, hence
 $$ k-m > \frac{2\frac{k(r+acos\frac{s}{r})}{R(s)^2} -  \frac{k}{r}}{\frac{n-1}{r} +\frac{k(r+acos\frac{s}{r})}{R(s)^2}}. $$ Since $\gamma'(s) \neq (0,1),$ $\sin (s\slash r)>0$, and it follows that $\kappa'(s)>0.$

\end{proof}

\begin{lemma}[{cf. [Ch, Lemma 3.5]}]\label{H''0} 
If $\gamma$ is not a centered circle, then $\kappa''(0)=0$.
\end{lemma}

\begin{proof}

For $s=0$, $C_0=A_0$, and the circle approximates $\gamma$ up to fourth order. As such, we have that 
$$\kappa''(C_s)= \kappa''(C_{\tilde s}),$$ 
$$H_1''(0)= H_1''(\tilde 0),$$ and 
$$\frac{d}{ds}\frac{g(s)}{f(s)}= \frac{d}{d \tilde s}\frac{g(\tilde s)}{f(\tilde s)}.$$ 
Observe that $C_0$ is not a centered circle, as $\kappa(0)=\kappa(C_0)$, and so by Chambers [Ch, Lemma 3.2] $\gamma$ would be a centered circle. Since we assumed that this is not the case, $C_0$ is not centered. Since $\kappa''(C_s)= \kappa''(C_{\tilde s}),$ $C_0$ approximates $\gamma$ up to fourth order, 
by Lemma \ref{H1}  we obtain
$$H_{f,g}''(0)= \kappa ''(0)R(s)^{k-m} - \frac{a}{r}(k-m)(n-1)\kappa(s)(a+r)^{k-m-2} $$ $$-\frac{ak}{r^2}R(s)^{k-m-2}-\frac{ak}{r}{k-m-2}(a+r)^3.$$
Since $a>0$, by Lemma \ref{H1''} $H''_{f,g} > 0$. 
\end{proof}

\begin{lemma} [{cf. [Ch, Lemma 3.7]}] \label{(-1,0)}
Let $s \in [0,\beta)$. If $\gamma '(s) = (-1,0)$, $\gamma_1(s)>0$, and $\kappa(s) \geq \kappa(C_s) > 0$, then $\kappa'(s) > 0$.
\end{lemma}
\begin{proof}
Since $\gamma$ is regular at $s$, the circle $A_s$ is defined and approximates $\gamma$ up to the third order at $s$. As such, $$\kappa'(C_s)= \kappa'(C_{\tilde s}).$$ 
Note that on $A_s$ $\kappa (C_{\tilde s})$ has a critical point at $\tilde s_0$  when  the tangent to $A_s$ at $\tilde s_0$ is $(\pm 1, 0)$.
Since $\gamma '(s) = (-1,0)$, we have that  $0=\kappa'(C_{\tilde s})=\kappa'(C_s)$. Since $A_s$ approximates $\gamma$ to third order, 
$$H_1'(s)=\tilde H_1(\tilde s).$$
Since $\gamma_1(s) >0$, the center of $C_s$ is greater than zero. Again, since $A_s$ approximates $\gamma$ to third order, 
$$ \frac{d}{ds}\frac{g(s)}{f( s)} =\frac{d}{d\tilde s} \frac{g(\tilde s)}{f(\tilde s)}.$$
By Lemma \ref{H1} 
$$H_{f,g}' = -a(k-m)((n-2)\kappa(C_s)+\kappa(s))R(s)^{k-m-2}\sin\frac{s}{r} $$ $$+ \kappa'(s)R(s)^{k-m} -\frac{ak}{r}\sin\frac{s}{r}R(s)^{k-m-2} $$ $$- a k(k-m-2)R(s)^{k-m-4}(r + a\cos\frac{s}{r})\sin\frac{s}{r}.$$ 
Since $H_{f,g}'=0$, 
$$\kappa'(s)= a\sin\frac{s}{r}R(s)^{-2}[(k-m)((n-2)\kappa(C_s)+\kappa(s)) $$ $$+ \frac{k(k-m-2)(r+a\cos\frac{s}{r})}{R(s)^2}+ \frac{k}{r}].$$ 
Since $a>0$,
 $$ k-m > \frac{2\frac{k(r+acos\frac{s}{r})}{R(s)^2} -  \frac{k}{r}}{\frac{n-1}{r} +\frac{k(r+acos\frac{s}{r})}{R(s)^2}}. $$ 
Since $\gamma'(s)=(-1,0),$ $ \sin (s\slash r)>0$, and it follows that $\kappa'(s)>0.$
\end{proof}

Lemmas \ref{3.14C} is the analog of Chambers [Ch, Lemma 3.14].

\begin{lemma}[{cf. [Ch, Lemma 3.14]}] \label{3.14C} 
If $\gamma$ is not a centered circle, then we have that $\nu < \beta$, and $\nu \in L$.
\end{lemma}
\begin{proof}
The proof proceeds as in Chambers [Ch, Lemma 3.14] except that we use Lemma \ref{curvaturez} to show that $\kappa (z) < \kappa(\overline{z})$  for small positive $\overline{z}$. %The assumptions of Lemma \ref{curvaturez} are justified in Chambers [Ch, Lemma 3.14].
\end{proof}

\begin{lemma} \label{curvaturez} 
Suppose $z, \overline{z} $ satisfy the following:
\begin{enumerate}

\item  $\abs{\gamma(z)} \leq \abs{\gamma(\overline{z})},$
\item  $N(\gamma(z) \cdot \gamma'(z)^{\perp} <N(\gamma(\overline{z}) \cdot, \gamma'(\overline{z})^{\perp} ,$
\item $\kappa(C_z) \leq  \kappa (C_{\overline{z}}).$
\end{enumerate} 
Then $\kappa(z) <\kappa (C_{\overline{z}}).$
\end{lemma}

\begin{proof}
The Lemma follows directly from the fact that $H_{f,g}(z) = H_{f,g}(\overline z)$ and that we can express  $H_{f,g}(u)$ as $$H_{f,g}(u) = r^{k-m}(\kappa(u)+(n-2)\kappa(C_u)) +kr^{k-m-1}N(\gamma(u) \cdot \gamma'(u)^{\perp}$$
\end{proof}

\newpage

 \bibliographystyle{alpha}

\bigskip
\newpage

Lewis \& Clark College 

\url{digiosia@lclark.edu}

\vspace{.5cm}

Williams College 

\url{jih1@williams.edu }

\vspace{.5cm}

Stony Brook University

\url{lea.kenigsberg@stonybrook.edu }

\vspace{.5cm}

Williams College

\url{dsp1@williams.edu }

\vspace{.5cm}

Williams College 

\url{wz1@williams.edu }

\end{document}